\documentclass[12pt]{amsart}
\usepackage{a4wide,enumerate,color,graphicx}
\usepackage{amsmath}
\usepackage{scrextend} % for labeling environment
\allowdisplaybreaks

\let\pa\partial
\let\na\nabla
\let\eps\varepsilon
\newcommand{\N}{{\mathbb N}}
\newcommand{\R}{{\mathbb R}}

\newcommand{\diver}{\operatorname{div}}

\newtheorem{theorem}{Theorem}
\newtheorem{lemma}[theorem]{Lemma}
\newtheorem{proposition}[theorem]{Proposition}

\begin{document}

\title[Fluid mixtures driven by a pressure gradient]{Analysis of cross-diffusion 
systems for fluid mixtures driven by a pressure gradient}

\author[P.-E. Druet]{Pierre-Etienne Druet}
\address{Weierstrass Institute, Mohrenstr. 39, 10117 Berlin, Germany}
\email{pierre-etienne.druet@wias-berlin.de}

\author[A. J\"ungel]{Ansgar J\"ungel}
\address{Institute for Analysis and Scientific Computing, Vienna University of
	Technology, Wiedner Hauptstra\ss e 8--10, 1040 Wien, Austria}
\email{juengel@tuwien.ac.at}

\date{\today}

\thanks{The first author was supported by the grant D1117/1-1 of the 
German Science Foundation (DFG). The second author acknowledges
partial support from the Austrian Science Fund (FWF), grants P30000, P33010,
W1245, and F65.
}

\begin{abstract}
The convective transport in a multicomponent isothermal compressible fluid
subject to the mass continuity equations is considered. The velocity is
proportional to the negative pressure gradient, according to Darcy's law, and
the pressure is defined by a state equation imposed by the volume extension of 
the mixture. These model assumptions lead to a parabolic-hyperbolic system 
for the mass densities.
The global-in-time existence of classical and weak solutions is proved in a 
bounded domain with no-penetration boundary conditions. The idea is to 
decompose the system into a porous-medium-type equation for the volume extension
and transport equations for the modified number fractions. The existence proof
is based on parabolic regularity theory, the theory of renormalized solutions,
and an approximation of the velocity field.
\end{abstract}

% \paragraph{Keywords:}

% \paragraph{AMS classification:}

\maketitle

%%%%%%%%%%%%%%%%%%%%%%%%%%%%%%%%%%%%%%%%%%%%%%%%%%%%%%%%%%%%%%%%%%%%%%%%%%%%%%%

\section{Introduction}

Multicomponent fluids are found in nature and many engineering applications,
for instance, in combustion, chemical reactors, tumor growth, and gas mixtures. 
For efficient modeling and simulations of these applications, we need 
to understand the mathematical structure of the
governing partial differential equations and to determine the properties of
their solutions. In this paper, we analyze the mass continuity equations for
the partial mass densities, subject to Darcy's law for the fluid velocity and
a pressure related to the volume extension of the mixture, and we prove the
global-in-time existence of smooth and weak solutions. The main novelty is 
the general framework of our pressure model.

\subsection{Model equations}

We consider the evolution of a fluid mixture of $N$ substances in a bounded
container $\Omega\subset\R^3$. We assume that the system is in the isothermal
state. The mass densities $\rho_1,\ldots,\rho_N$
of the species are given by the conservation equations
\begin{equation}\label{1.cont}
  \pa_t\rho_i + \diver(\rho_i v) = 0 \quad\mbox{in }\Omega\times(0,T),\ i=1,\ldots,N.
\end{equation}
We suppose that the fluid is driven only by the thermodynamic pressure. Then 
the barycentric velocity $v$ is determined by Darcy's law
\begin{equation}\label{1.v}
  v = -\kappa \na p\quad\mbox{in }\Omega\times(0,T),
\end{equation}
where the porosity coefficient $\kappa$ generally depends on the medium $\Omega$ and
the fluid. The pressure is related to the volume extension $\Lambda$ 
of the mixture via
\begin{equation}\label{1.p}
  p = G(\Lambda(\rho)), \quad \rho=(\rho_1,\ldots,\rho_N),
\end{equation}
where $G$ is an increasing scalar function and $\Lambda$ is positively homogeneous
(of degree one). Typical choices are $G(s)=c_0s^\alpha$ with $\alpha>1$ and 
$\Lambda(\rho)=\sum_{i=1}^Nc_i\rho_i$, where $c_0,\ldots,c_N>0$.
We refer to Section \ref{sec.model} for the modeling details.

Equations \eqref{1.cont}--\eqref{1.p} can be formulated as a cross-diffusion system
with entropy structure. Indeed, we show in Section \ref{sec.model} that there exists
a free energy (or entropy) $h(\rho)$ such that $\na p=\sum_{j=1}^N\rho_j\na\mu_j$, 
where $\mu_j=\pa h/\pa\rho_j$ are the chemical potentials (or entropy variables).
Then system \eqref{1.cont}--\eqref{1.p} is equivalent to the cross-diffusion system
\begin{equation}\label{1.rhomu}
  \pa_t\rho_i - \diver\sum_{j=1}^N M_{ij}(\rho)\na\mu_j = 0, \quad i=1,\ldots,N,
\end{equation}
where the kinetic (or mobility) matrix 
$M_{ij}(\rho)=\kappa\rho_i\rho_j$ has rank one only. 
Thus, we expect that there is only one ``parabolic direction'' and
$N-1$ ``hyperbolic directions''. The challenge is to deal with such cross-diffusion
systems possessing {\em incomplete diffusion}. It is sufficient to
impose one boundary condition, and we choose in this paper the no-penetration 
(and initial) conditions
\begin{equation}\label{1.bic}
  v\cdot\nu = 0 \quad\mbox{on }\pa\Omega\times(0,T), \quad
	\rho_i(0) = \rho_i^0\quad\mbox{in }\Omega,\ i=1,\ldots,N,
\end{equation}
where $\nu$ is the exterior unit normal vector to $\pa\Omega$. Alternatively,
in the case that there are free inflows and outflows, we may use the pressure
boundary condition
$$
  p = p_{0}\quad\mbox{on }\pa\Omega\times(0,T).
$$
In practical situations, a mix of free flow and impermeable portions of the boundary 
is often realistic. This leads to more technical problems in the mathematical 
analysis, so we consider mainly the boundary conditions in \eqref{1.bic}. Some 
remarks to treat pressure boundary conditions are given in 
Section \ref{sec.bc}.

\subsection{State of the art}

Our study is motivated by some related problems. First, in \cite{JMZ18}, we have
analyzed the cross-diffusion system
\begin{equation}\label{1.waals}
  \pa_t\rho_i - \diver\sum_{j=1}^N\big(\rho_i\rho_j + \eps D_{ij}(\rho)\big)\na\mu_j = 0,
	\quad i=1,\ldots,N,
\end{equation}
where the pressure models a Van der Waals gas mixture (this determines $\mu_j$)
and the matrix $(D_{ij}(\rho))$ is positive definite on the orthogonal complement 
of a one-dimensional subspace of $\R^N$. 
Because of the lack of parabolicity, the diffusion fluxes
$J_i=-\eps\sum_{j=1}^N D_{ij}(\rho)\na\mu_j$ with $\eps>0$ were needed to apply
the boundedness-by-entropy method \cite{Jue15}. Unfortunately, the authors
of \cite{JMZ18} were not able to perform the limit $\eps\to 0$. In this paper,
we show that \eqref{1.waals} admits solutions also in the case $\eps=0$
(for suitable pressure functions).

Second, consider $N$ interacting biological species with densities $\rho_i$
and velocities $v_i$ governed by the continuity equations 
$$
  \pa_t\rho_i + \diver(\rho_i v_i) = 0, \quad i=1,\ldots,N.
$$
When dispersal is a response to population pressure, we may assume that the
dispersal of each of the species is driven by the gradient of the total population
\cite{BGHP85,GuPi84},
$$
  v_i = -k_i\na\sum_{j=1}^N\rho_j, \quad i=1,\ldots,N.
$$
This leads to system \eqref{1.rhomu} with $M_{ij}=\rho_i\rho_j$ and
$\mu_j=k_j\log\rho_j$. The model describes the evolution of the cell populations
$\rho_i$ in tissues and tumors. It was analyzed for two species $N=2$ in \cite{BGHP85}
in one space dimension and in \cite{BHIM12} in several space dimensions.
The existence of global weak and classical solutions was proved. They possess
the particular feature that they are segregated if they do so initially,
i.e., their support is disjoint for all time. For related models, we refer to
\cite{GaSe14,GSV15}.

A nonlinear pressure function $p=(\rho_1+\rho_2)^\alpha$ with $\alpha>1$ was
considered in \cite{GPS19}, giving \eqref{1.rhomu} with $N=2$,
$M_{ij}=\rho_i\rho_j(\rho_1+\rho_2)^{\alpha-1}$, and $\mu_j=\log\rho_j$.
The existence of weak solutions in the whole space was proved in the presence of
reaction terms. The two-species system of Bertsch et a.\ \cite{BHIM12}
with nonlocal interaction terms models aggregation and repulsion of the species
in the context of chemotaxis, opinion formation, and pedestrian dynamics
(see the references in \cite{BDFS18}). 

If the pressure is the variational derivative of a certain energy functional $F(\rho)$, 
we may write \eqref{1.cont} for $N=2$ as the (formal) gradient flow
$$
  \pa_t\rho_1 = \diver\bigg(\rho_1\na\frac{\delta F}{\delta\rho_1}\bigg), \quad
	\pa_t\rho_2 = \diver\bigg(\rho_2\na\frac{\delta F}{\delta\rho_2}\bigg).
$$
This relation has been exploited in \cite{BCPS19,DEF18}, proving the convergence
of the minimizing movement scheme for quadratic functionals $F$ \cite{BCPS19}
or general convex functionals \cite{DEF18}, even including nonlocal terms; also see
\cite{CFSS18}. 

Related systems that consist of the mass continuity equation 
\eqref{1.cont} and the Darcy law \eqref{1.v} are analyzed in the literature also
in the context of fluid flows in porous media \cite{LDQ11}, often extended
by the Darcy--Brinkman law for fractures in porous media \cite{MoSh17} 
or for tumor growth models \cite{GLNS18}.

Surprisingly, there are almost no results for general $N$-species models.
The work \cite{CoUg16} studies \eqref{1.rhomu}
with $M_{ij}=\rho_i\rho_j/\rho_{\rm tot}$, where $\rho_{\rm tot}=\sum_{i=1}^N\rho_i$,
and $\mu_j=\log\rho_j$. Furthermore, system \eqref{1.rhomu} with 
$M_{ij}=a_{ij}\rho_i\rho_j$ and $a_{ij}>0$ is the mean-field limit of an 
interacting stochastic particle system \cite{CDJ19}. Up to our knowledge, there
exists no general global existence result for \eqref{1.rhomu}.

In this paper, we analyze the $N$-species system with a rather general 
pressure-density relation, extending all previous existence results. 
In \cite{BHIM12,GuPi84}, the problem is decomposed into a porous-medium-type equation 
for the total density $\rho_{\rm tot}$ (in our situation: $\Lambda(\rho)$)
and transport equations for the mass fractions $\rho_i/\rho_{\rm tot}$, $i=1,\ldots,N$ 
(in our situation: $\rho_i/\Lambda(\rho)$).
Compared to previous work, the relationship \eqref{1.p} between the pressure and the
densities is more general, we allow for an arbitrary number of species, and
we prove the uniqueness of weak and classical solutions. Special effort is
necessary to allow for partially vanishing initial total densities \cite{CFSS18,GPS19};
in the context of fluid dynamics, however, initial vacuum would not be meaningful.

\subsection{Key ideas}\label{keyideas}

As mentioned above, the idea is to decompose the equations in one parabolic
equation for the function $\Lambda(\rho)$ and $N$ transport equations
for the variables $u_i=\rho_i/\Lambda(\rho)$.

First, we derive the parabolic equation.
Let $(\rho,v)$ be a differentiable solution to \eqref{1.cont}--\eqref{1.p}.
Multiplying \eqref{1.cont} by $\pa\Lambda/\pa\rho_i$ and summing over $i=1,\ldots,N$
leads to
$$
  \pa_t\Lambda(\rho) + v\cdot\na\Lambda(\rho) + \Lambda'(\rho)\cdot\rho
	\diver v = 0.
$$
The positive homogeneity of $\Lambda$ implies that $\Lambda'(\rho)\cdot\rho
=\Lambda(\rho)$, so $\Lambda$ solves the conservation law
$\pa_t\Lambda(\rho)+\diver(\Lambda(\rho)v)=0$. Then we infer from
$v=-\kappa\na p=-\kappa G'(\Lambda(\rho))\na\Lambda(\rho)$ that the variable
$w:=\Lambda(\rho)$ solves the nonlinear diffusion equation
\begin{equation}\label{1.w}
  \pa_t w - \diver(\kappa w G'(w)\na w) = 0\quad\mbox{in }\Omega\times(0,T),
\end{equation}
together with the initial and boundary conditions
\begin{equation}\label{1.bicw}
  w(0) = w^0 := \Lambda(\rho^0)\quad\mbox{in }\Omega, \quad
	\na w\cdot\nu = 0\quad\mbox{on }\pa\Omega\times(0,T).
\end{equation}

We claim that the variable $u_i=\rho_i/\Lambda(\rho)$ ($i=1,\ldots,N$), 
which can be interpreted
as a kind of volume fraction, solves a transport equation. Indeed, with the
material derivative $\dot{u}=(\pa_t+v\cdot\na)u$, it follows that $\rho_i$ solves
$\dot\rho_i = -\rho_i\diver v$. We use the continuity
equation \eqref{1.cont} and the identity $\Lambda'(\rho)\cdot\rho=\Lambda(\rho)$ 
to compute
$$
  \dot{u}_i = \frac{\dot\rho_i}{\Lambda(\rho)} - \frac{\rho_i}{\Lambda(\rho)^2}
	\Lambda'(\rho)\cdot\dot\rho
	= \frac{\dot\rho_i}{\Lambda(\rho)} + \frac{\rho_i}{\Lambda(\rho)^2}
	\Lambda'(\rho)\cdot\rho\diver v = 0.
$$
Thus, the volume fractions $u_i$ are just transported:
\begin{equation}\label{1.u}
  \dot{u}_i = 0 \quad\mbox{in }\Omega\times(0,T),
	\quad u_i(0)=u_i^0:=\frac{\rho_i^0}{\Lambda(\rho^0)}
	\quad\mbox{in }\Omega,\ i=1,\ldots,N.
\end{equation}

The nonlinear diffusion equation \eqref{1.w} is solved by standard
techniques: The weak maximum principle yields positive lower and upper bounds
for the solution $w$, and parabolic regularity theory provides 
further a priori estimates.
The velocity is then given by $v=-\kappa G'(w)\na w$. Because of its low regularity,
we approximate the velocity field by smooth functions $v_\eps$, such that
we can solve the transport equation \eqref{1.u} with $v$ replaced by $v_\eps$.
This yields the solutions $u_i^\eps$ and the approximate densities
$\rho_i^\eps := wu_i^\eps$ for $i=1,\ldots,N$. The properties of the approximate
velocity fields allow us to prove that $\rho_i^\eps$ converges to a weak solution
$\rho_i$ to the mass continuity equation as $\eps\to 0$. 
For weak solutions and in particular for the proof of $\Lambda(\rho)=w$, 
the use of renormalization techniques is necessary, since 
we want to solve the continuity equations with a velocity field that possesses 
only local Sobolev regularity.

\subsection{Notation}

We set $\R_+=(0,\infty)$ and $\R_{+,0}=[0,\infty)$.
The space $C^{k+\alpha}(\Omega)$ for $k\in\N$, $\alpha\in(0,1]$
consists of all functions $u\in C^k(\overline\Omega)$ whose $k$th 
partial derivatives are
H\"older continuous of order $\alpha$ up to the boundary of $\Omega$. 
Set $Q_T:=\Omega\times(0,T)$ and $\overline{Q}_T:=\overline\Omega\times[0,T]$.
The space $C^{k+\alpha,\ell+\beta}(\overline{Q}_T)$ consists of all functions which
are $C^{k+\alpha}$ with respect to the spatial variable and $C^{\ell+\beta}$ with
respect to the time variable, where $k$, $\ell\in\N$ and $\alpha$, $\beta\in(0,1]$.

\subsection{Main results}

We make the following assumptions.

\begin{labeling}{(A44)}
\item[(A1)] Domain: $\Omega\subset\R^3$ is a bounded domain with Lipschitz boundary.
\item[(A2)] Initial datum: $\rho^0=(\rho_1^0,\ldots,\rho_N^0)\in L^\infty(\Omega;\R^N)$ 
satisfies $p(\rho^0):= G(\Lambda(\rho^0)) \in H^1(\Omega)$, $\rho_i^0\ge 0$ in 
$\Omega$ for $i=1,\ldots,N$, 
and $\sum_{i=1}^N\rho_i^0\ge c_0>0$ in $\Omega$ for some $c_0>0$.
\item[(A3)] Function $G$: $G\in C^1(\R_+)$ is strictly increasing, i.e.\ $G'(s)>0$
for all $s>0$, and $s\mapsto sG'(s)$ is of class $C^2(\R_+)$.
\item[(A4)] Function $\Lambda$: $\Lambda\in C^2(\R_+^N)\cap C^0(\R^N_{+,0})$
is nonnegative, convex, and positively homogeneous (of degree one) and there exist
constants $0<r_0<r_1<\infty$ such that
$$
  r_0|\rho| \le\Lambda(\rho)\le r_1|\rho| \quad\mbox{for all }\rho\in\R_+^N.
$$
\end{labeling}

Assumptions (A1) and (A2) are rather natural. We already mentioned before that
partially vanishing total densities are not meaningful in fluid dynamics,
and we require in Assumption (A2) that $\sum_{i=1}^N\rho_i^0$ is strictly positive.
The variable $w=\Lambda(\rho)$ satisfies the porous-medium-type equation \eqref{1.w},
and the condition that $s\mapsto sG'(s)$ is of class $C^2(\R_+)$ in Assumption
(A3) is needed to deduce classical solutions; see Section \ref{sec.class}.
The hypotheses on the volume extension $\Lambda$ in Assumption (A4) guarantee
that the variable $u_i=\rho_i/\Lambda(\rho)$ satisfies the transport equation
\eqref{1.u}. Moreover, the linear growth condition simplifies some estimates
in Section \ref{sec.weak}.

A possible obvious extension would be to consider porosity coefficients depending
on the density or pressure, $\kappa = \kappa(\Lambda(\rho))$ or $\kappa = \kappa(p)$.
We comment briefly in Section \ref{sec.open} how Assumption (A3) can be
suitably modified to treat this case.

We next formulate our main theorems on the well-posedness analysis.

\begin{theorem}[Classical solutions]\label{thm.class}
Let Assumptions (A1)--(A4) hold, let $T>0$, and $\kappa>0$. Furthermore, let
$\pa\Omega\in C^{2+\alpha}$ for some $\alpha>0$, $G\in C^2(\R_+)$, and
$\rho^0\in C^{1+\alpha}(\overline\Omega;\R^N)$ such that $p(\rho^0)\in 
C^{2+\alpha}(\overline\Omega)$ and $\na p(\rho^0)\cdot\nu=0$ on $\pa\Omega$. 
Then there exists a unique classical solution $\rho\in C^{1+\alpha,1}
(\overline{Q}_T;\R_{+,0}^N)$ with $p(\rho)\in C^{2+\alpha,1+\alpha/2}
(\overline{Q}_T)$ to the equations
$$
  \pa_t\rho_i - \diver(\kappa\rho_i\na p(\rho)) = 0\quad\mbox{in }\Omega\times(0,T),\
	i=1,\ldots,N,
$$
and the initial and boundary conditions \eqref{1.bic} are satisfied.
\end{theorem}

\begin{theorem}[Weak solutions]\label{thm.weak}
Let Assumptions (A1)--(A4) hold, let $T>0$, and $\kappa>0$. 
Furthermore, let $\pa\Omega$ be piecewise of class $C^2$ (i.e., $\Omega$ is a
curvilinear polyhedron). Then there exists a weak solution $(\rho,v)$ 
to \eqref{1.cont}--\eqref{1.p} and \eqref{1.bic} such that
$\rho\ge 0$ in $Q_T$, $\rho\in L^\infty(Q_T;\R^N)$, $v=-\kappa\na p(\rho)\in L^2(Q_T)$,
and $\diver v\in L^\sigma(Q_T)$ for some $\sigma>4/3$, and $\rho$ 
is the unique weak solution to \eqref{1.cont} and \eqref{1.bic}.
\end{theorem}

The paper is organized as follows. The pressure relation \eqref{1.p} and
the cross-diffusion formulation \eqref{1.rhomu} is motivated in Section
\ref{sec.model}. Then the proofs of Theorems \ref{thm.class} and \ref{thm.weak}
are presented in Sections \ref{sec.class} and \ref{sec.weak}, respectively. 
Finally, we collect in Section \ref{sec.open} some possible extensions of our 
results and open problems.

%%%%%%%%%%%%%%%%%%%%%%%%%%%%%%%%%%%%%%%%%%%%%%%%%%%%%%%%%%%%%%%%%%%%%%%%%%%%%

\section{Modeling}\label{sec.model}

In this section, we motivate system \eqref{1.cont}--\eqref{1.p}. We consider
$N$ chemical substances with mass densities $\rho_i$ whose isothermal evolution
is governed by the continuity equations \eqref{1.cont}, subject to Darcy's
law \eqref{1.v}. We introduce the partial number densities $n_i=\rho_i/m_i$,
where the molecular masses $m_1,\ldots,m_N$ are positive constants. Then
$n_{\rm tot}:=\sum_{i=1}^Nn_i$ is the total number density, and
$X_i:=n_i/n_{\rm tot}$ are the number fractions.

In order to model the thermodynamic pressure $p$, we follow some ideas exposed in 
\cite[Section 15]{BoDr15} to describe the free energy of elastic mixtures. We also
refer to \cite{DGM13,DGM18} for further particular models in the case of mixtures 
with charged carriers. The pressure is related to the volume extension of the mixture 
via a state equation of the following kind:
\begin{equation}\label{2.p}
  p = G\big(n_{\rm tot}H(X_1,\ldots,X_N)\big),
\end{equation}
where $G$ and $H$ are suitable functions. The choice $G(s)=s$, $H=1$ gives
the pressure law of (isothermal) ideal gases, namely $p=n_{\rm tot}$. 
Other special choices are $G(s)=s^\alpha$ with $\alpha>1$ and $H=1$,
giving $p = n_{\rm tot}^{\alpha}$, while for $G(s)=s^\alpha$, 
$H(X)=(\sum_{i=1}^NX_i^\alpha)^{1/\alpha}$, we obtain $p = \sum_{i=1}^N n_i^{\alpha}$. 
We refer to \cite{Li67} for more examples of nonlinear state equations,
often refered to as Tait equations.
We notice that in the last example $p = \sum_{i=1}^N n_i^{\alpha}$, there 
is possibly a relationship to the Dalton law with partial pressures of the 
constituents obeying $p_i=n_i^\alpha$ with uniform exponent $\alpha$. 
This seems to be a by-product of the positive homogenity assumption for $H$,
for suitably chosen functions $G$. In general, the state equation \eqref{2.p} 
yields models that are not equivalent to the Dalton law.
The factor $H$ which represents an average volume can be used to model 
finite-volume effects of the molecules, for instance with the linear ansatz 
$H=\sum_{i=1}^N V_iX_i$ ($V_i$ are the reference partial volumes, like in 
\cite{DGL14}). 

To simplify our notation, we express
the pressure law for the mass densities. Indeed, noting that we can identify
each function of $n_1,\ldots,n_N$ with a function of $\rho_1,\ldots,\rho_N$ via
$\rho_i=m_in_i$, we express \eqref{2.p} in the equivalent form
$p=G(\Lambda(\rho_1,\ldots,\rho_N))$, where
$$
  \Lambda(\rho_1,\ldots,\rho_N) := n_{\rm tot}H\bigg(\frac{n_1}{n_{\rm tot}},\ldots,
	\frac{n_N}{n_{\rm tot}}\bigg)
$$
describes the volume extension of the mixture. 
It is essential for our analysis that the mapping $\rho\mapsto\Lambda(\rho)$ 
is positively homogeneous of degree one. This condition is always satisfied 
in applications, since the function $H$ depends only on the fractions
$n_i/n_{\rm tot}$ which are homogeneous of degree zero in $\rho$.

We claim that system \eqref{1.cont}--\eqref{1.p} can be formulated as the
cross-diffusion system \eqref{1.rhomu}. To this end, we introduce the scalar function
$$
  h_M(s) = s\int_{s_0}^s\frac{G(\tau)}{\tau^2}d\tau + \mbox{const.}, \quad s>0,
$$
and the free energy $h(\rho)=h_M(\Lambda(\rho))$ for $\rho\in[0,\infty)^N$.
The Hessian of $h$,
$$
  D^2 h = h_M''(D\Lambda\otimes D\Lambda) + h_M'D^2\Lambda,
$$
is positive definite if $h_M$ is increasing and convex and $\Lambda$ is convex. 
For instance, we may assume that $G$ is increasing, which leads to 
$h_M''(s)=G'(s)/s>0$. The pressure is given by the Gibbs--Duhem relation
$$
  p = -h + \sum_{i=1}^N\rho_i\mu_i,
$$
where $\mu_i:=\pa h/\pa\rho_i$ are the chemical potentials. Since
$sh_M'(s)-h_M(s)=G(s)$ and $\Lambda$ is supposed to be positively homogeneous
(implying that $\Lambda'(\rho)\cdot\rho=\Lambda(\rho)$), we find that
$$
  p = -h_M(\Lambda(\rho)) + \sum_{i=1}^N\rho_i\frac{\pa\Lambda}{\pa\rho_i}(\rho)
	h_M'(\Lambda(\rho)) = -h_M(\Lambda(\rho)) + h_M'(\Lambda(\rho))\Lambda(\rho)
	= G(\Lambda(\rho)),
$$
which equals \eqref{1.p}. Furthermore, the Gibbs--Duhem relation implies that
$\na p= \sum_{j=1}^N\rho_j\na\mu_j$, and inserting this expression into 
\eqref{1.cont}--\eqref{1.v} leads to the cross-diffusion system \eqref{1.rhomu}.

%%%%%%%%%%%%%%%%%%%%%%%%%%%%%%%%%%%%%%%%%%%%%%%%%%%%%%%%%%%%%%%%%%%%%%%%%%%%%

\section{Proof of Theorem \ref{thm.class}}\label{sec.class}

Problem \eqref{1.w}--\eqref{1.bicw} is a quasilinear parabolic equation with
Neumann boundary conditions. Its unique solvability in $C^{2+\alpha,1+\alpha/2}
(\overline{Q}_T)$ follows from Theorem 7.4 in \cite[Chapter V.7]{LSU68} 
(also see \cite[Theorem 10.24]{FeNo09}) if the following conditions are satisfied:
\begin{itemize}
\item Every classical solution to \eqref{1.w}--\eqref{1.bicw} is bounded from above
and below;
\item the coefficient $\kappa wG'(w)$ is of class $C^2(\R_+)$;
\item the initial datum satisfies $w^0:=\Lambda(\rho^0)\in 
C^{2+\alpha}(\overline\Omega)$.
\end{itemize}
These conditions are satisfied. Indeed, 
the weak maximum principle for parabolic equations shows that
\begin{equation}\label{2.Lambda}
  0 < \Lambda_* := \inf_{x\in\Omega}\Lambda(\rho^0(x))\le w(x,t)
	\le \Lambda^* := \sup_{x\in\Omega}\Lambda(\rho^0(x))\quad\mbox{for all }(x,t)\in Q_T,
\end{equation}
where $\Lambda_* = r_0c_0>0$, according to Assumptions (A2) and (A4).
The second condition follows from Assumption (A3) and the third condition
is a consequence of the assumptions of the theorem. Thus, there exists a unique solution
$w\in C^{2+\alpha,1+\alpha/2}(\overline{Q}_T)$ to \eqref{1.w}--\eqref{1.bicw}.

Next, we solve \eqref{1.u}. The transport velocity $v=-\kappa G'(w)\na w
=-\kappa \na G(w)$ is bounded
in $L^\infty(Q_T)$ and $\diver v$ is bounded in $C^{\alpha}(\overline{Q}_T)$.
This allows us to find for $x\in\Omega$, $t\in[0,T]$ the characteristic curves of
$$
  \Phi'(s;t,x) = v(\Phi(s;t,x), \, s),\quad s\in(0,T),\quad \Phi(t;t,x) = x.
$$
For all $s\in[0,T]$, the map $(x,t)\mapsto \Phi(s;x,t)$ belongs to
$C^{1+\alpha,1+\alpha/2}(\overline{Q}_T)$. Thus, problem \eqref{1.u} admits the
unique solutions
$$
  u_i(x,t) = u_i^0(\Phi(0;t,x)), \quad i=1,\ldots,N.
$$
Since $u^0=\rho^0/\Lambda(\rho^0)\in C^1(\overline\Omega;\R^N)$, 
we have $u\in C^1(\overline{Q}_T;\R^N)$.

We claim that $\rho = wu$ is a classical solution to \eqref{1.cont}, \eqref{1.bic}.
The positive homogeneity of $\Lambda$ shows that $\Lambda(\rho)=w\Lambda(u)$.
Moreover, $(\Lambda(u))^{\boldsymbol{\cdot}} = \Lambda'(u)\dot{u} = 0$ and
$\Lambda(u(0))=\Lambda(u^0)=1$, by definition of $u^0$. We conclude that
$\Lambda(u)=1$. A more direct proof of this property is as follows:
We multiply the equation $0=\dot{u} = \pa_t u + v\cdot\na u$ by $\Lambda'(u)$,
which yields $\pa_t\Lambda(u)+v\cdot\Lambda(u)=0$. Then we compute
\begin{align*}
  \frac{d}{dt}\int_\Omega(\Lambda(u)-1)^2dx
	&= -2\int_\Omega (\Lambda(u)-1)v\cdot\na\Lambda(u)dx
	= -\int_\Omega v\cdot\na(\Lambda(u)-1)^2 dx \\
	&= \int_\Omega\diver(v)(\Lambda(u)-1)^2 dx,
\end{align*}
and the regularity $\diver v\in L^1(0,T;L^\infty(\Omega))$ is sufficient to
conclude fron Gronwall's lemma that $\Lambda(u)=1$. We infer that
$\Lambda(\rho)=w\Lambda(u)=w$. We know that $w$ solves 
$0 = \pa_t w - \diver(\kappa w\na G(w)) = \pa_t w + \diver(wv) = \dot{w} + w\diver v$.
Then we deduce from
$$
  \dot{\rho}_i = \dot{w}u_i + w\dot{u}_i = \dot{w}u_i, \quad i=1,\ldots,N,
$$
and $\dot{w}=-w\diver v$ that
$$
  \dot{\rho}_i = \dot{w}u_i = -wu_i\diver v = -\rho_i\diver v,
$$
which is \eqref{1.cont} with $v=-\kappa\na G(w) = -\kappa\na G(\Lambda(\rho))$.
The initial condition in \eqref{1.bic} is satisfied since
$$
  \rho_i(0) = w(0)u_i(0) = w^0u_i^0 = \Lambda(\rho^0)\frac{\rho_i^0}{\Lambda(\rho^0)}
	= \rho_i^0\quad\mbox{in }\Omega.
$$
This finishes the proof of Theorem \ref{thm.class}.

%%%%%%%%%%%%%%%%%%%%%%%%%%%%%%%%%%%%%%%%%%%%%%%%%%%%%%%%%%%%%%%%%%%%%%%%%%%%%

\section{Proof of Theorem \ref{thm.weak}}\label{sec.weak}

We show first the following technical approximation property. It is needed to
construct smooth approximations of the velocity field and to identify the
weak limit of the density vector.

\begin{lemma}\label{lem.approx}
Let the assumptions of Theorem \ref{thm.weak} hold. Then there exists a family of 
functions $(\phi_m)_{m>0}\subset C^2(\overline\Omega)$ such that $0\le\phi_m\le 1$ in
$\Omega$, $\phi_m(x)=0$ for all $x\in\Omega$ with 
$\operatorname{dist}(x,\pa\Omega)<1/m$, $\phi_m\to 1$ locally uniformly in 
$\Omega$ as $m\to\infty$, and satisfying the following
property: For all $0<\delta<1/2$, there exists $C(\delta)>0$ such that
for all $g\in H^{1/2+\delta}(\Omega;\R^3)$ satisfying $g\cdot\nu=0$ on $\pa\Omega$
and for any $1\le p< 1/(1-\delta)$,
$$
  \|g\cdot\na\phi_m\|_{L^p(\Omega)}\le C(\delta)\|g\|_{H^{1/2+\delta}(\Omega)}.
$$
\end{lemma}

\begin{proof}
Since $\pa\Omega$ consists of surfaces $S_1,\ldots,S_\ell$ of class $C^2$,
we may assume that 
the distance function $d_i(x):=\operatorname{dist}(x,S_i)$ for $x\in\Omega$,
$i=1,\ldots,\ell$ is also of class $C^2$, i.e.\ $d_i\in C^2(\overline\Omega)$.
Let $m>0$ and introduce a function $h\in C^2(\R_+)$ such that $0\le h\le 1$
in $\R_+$, $h_m(s)=1$ if $s>2/m$, $h_m(s)=0$ for $0\le s<1/m$, and $h_m'(s)\le 3m$
for $s\in\R_+$. Then $h_m'(s)\le 3m\le 6/s$ for $1/m\le s\le 2/m$ and, in fact,
for all $s\in\R_+$ (since $h_m'=0$ otherwise).

We define $\phi_m(x):=\prod_{i=1}^\ell h_m(d_i(x))$. By definition of $h_m$, we have
$\phi_m(x)=0$ for $x\in\operatorname{dist}(x,\pa\Omega)<1/m$. The gradient 
of $\phi_m$ is given by
$$
  \na\phi_m = \sum_{k=1}^\ell h'_m(d_k)\na d_k\prod_{i\neq k}h_m(d_i).
$$
Using $h_m'(d_k) \le 6d_k^{-1}$ and $h_m\le 1$, we find that
\begin{equation}\label{3.aux2}
  |g\cdot\na\phi_m| \le 6\sum_{k=1}^\ell d_k^{-1}|g\cdot\na d_k|.
\end{equation}

We now estimate the right-hand side.
The properties of the distance function imply that $\nu(x)=\na d_i(x)$ for
$x\in S_i$. Moreover, as $d_i\in C^2(\overline\Omega)$,
\begin{align}
  \|g\cdot\na d_i\|_{H^{1/2+\delta}(\Omega)}
	&\le \|\na d_i\|_{L^\infty(\Omega)}\|g\|_{H^{1/2+\delta}(\Omega)}
	+ \|D^2 d_i\|_{L^\infty(\Omega)}\|g\|_{L^2(\Omega)} \nonumber \\
	&\le C\|g\|_{H^{1/2+\delta}(\Omega)}. \label{3.aux}
\end{align}
The interior trace operator $\operatorname{tr}_i:H^{1/2+\delta}(\Omega)
\to H^\delta(S_i)$, $\operatorname{tr}_i(v)=v|_{S_i}$, is bounded, 
so the boundary condition $g\cdot\nu=0$ on $S_i$ is valid almost everywhere.
This shows that
$$
  g\cdot\na d_i\in H^{1/2+\delta}_{S_i}(\Omega)
	:= \big\{f\in H^{1/2+\delta}(\Omega):\operatorname{tr}_i f=0
	\mbox{ in }H^\delta(S_i)\big\}, \quad i=1,\ldots,\ell.
$$
Theorem 11.3 in \cite[Chapter 1, \S 11]{LiMa72} states that the linear mapping
$H^{s}_{S_i}(\Omega)\to L^2(\Omega)$, $u\mapsto d_i^{-s}u$, with $s=1/2+\delta$
is continuous. This property and estimate \eqref{3.aux} imply that
\begin{equation}\label{3.aux3}
  \|d_i^{-s}g\cdot\na d_i\|_{L^2(\Omega)}
	\le C\|g\cdot\na d_i\|_{H^s(\Omega)} \le C\|g\|_{H^{1/2+\delta}(\Omega)}.
\end{equation}
Consequently, using \eqref{3.aux2} and the H\"older inequality, we obtain for 
$1\le p<2/(3-2s)=1/(1-\delta)$,
\begin{align*}
  \int_\Omega|g\cdot\na\phi_m|^p dx
	&\le C\sum_{k=1}^\ell\int_\Omega d_k^{-ps}|g\cdot\na d_k|^p d_k^{-p(1-s)} dx \\
	&\le C\sum_{k=1}^\ell\bigg(\int_\Omega d_k^{-2s}|g\cdot\na d_k|^2 dx\bigg)^{p/2}
	\bigg(\int_\Omega d_k^{-2p(1-s)/(2-p)}dx\bigg)^{1-p/2}.
\end{align*} 
It holds that $-2p(1-s)/(2-p)>-1$ if and only if $p<1/(1-\delta)$, and under this
condition, the last integral is finite. Thus, we infer from \eqref{3.aux3} that
$$
  \int_\Omega|g\cdot\na\phi_m|^p dx \le C(\delta)\sum_{k=1}^\ell
	\|d_k^{-s} g\cdot\na d_k\|_{L^2(\Omega)}^p 
	\le C(\delta)\|g\|_{H^{1/2+\delta}(\Omega)}^p,
$$
which finishes the proof. 
\end{proof}

We formulate \eqref{1.w} as
\begin{equation}\label{3.w}
  \pa_t w - \diver(a(w)\na w)=0 \quad\mbox{in }\Omega\times(0,T),
\end{equation}
where $a(w):=\kappa wG'(w)$ for $w\ge 0$, and we set
\begin{equation}\label{1.W12}
  W_2^1(Q_T) := \big\{f\in L^2(0,T;H^1(\Omega)):\pa_t f\in L^2(0,T;H^1(\Omega)')\big\}.
\end{equation}
Observe that we may allow for suitable porosity coefficients $\kappa=\kappa(w)$
at this point.
The following result shows that problem \eqref{1.w}--\eqref{1.bicw} is uniquely
solvable and that the vector field $v=-\kappa G'(w)\na w$ has some regularity
properties.

\begin{proposition}[Existence and regularity for \eqref{1.w}]\label{prop.v}
Under the assumptions of Theorem \ref{thm.weak}, problem \eqref{1.w}--\eqref{1.bicw}
has a unique solution $w\in W_2^1(Q_T)\cap L^\infty(0,T;L^\infty(\Omega))$
satisfying \eqref{2.Lambda}. Moreover, $v=-\kappa G(w)\na w$ satisfies the following
properties:
\begin{itemize}
\item $v\in L^{4/3}(0,T;H^1(\Omega'))$ for any domain $\Omega'$ compactly included
in $\Omega$;
\item $v\in L^2(0,T;L^q(\Omega))$ and $\diver v\in L^{2-2/q}(Q_T)$ for some 
$3 < q\le 6$;
\item $v\cdot\na\phi_m\to 0$ strongly in $L^1(Q_T)$ as $m\to\infty$ for $(\phi_m)$
constructed in Lemma \ref{lem.approx}.
\end{itemize}
\end{proposition}

\begin{proof}
{\em Step 1: Existence for \eqref{1.w}.} 
The existence of weak solutions to \eqref{1.w}--\eqref{1.bicw} 
can be shown by a standard approximation procedure, but 
since $a(w)$ does not need to be monotone, we need to be careful and therefore
present a sketch of the proof. We introduce the truncated functions
$$
  a_k(s) := \left\{\begin{array}{ll}
	a(k) & \quad\mbox{if }s\ge k, \\
	a(s) & \quad\mbox{if }1/k<s<k, \\
	a(1/k) & \quad\mbox{if }0\le s\le 1/k,
	\end{array}\right.
$$
where $k\in\N$. The functions $B_k:\R_+\to\R_+$,
$B_k(s):=\int_0^s a_k(z)dz$, are bi-Lipschitz transformations. 
We consider the approximated problem
\begin{equation}\label{3.u}
  \frac{1}{a_k(B_k^{-1}(u))}\pa_t u = \Delta u, \quad
	\na u\cdot\nu=0\quad\mbox{on }\pa\Omega, \quad u(0)=B_k(w^0)\quad\mbox{in }\Omega.
\end{equation}
If a solution to this problem is given, then $w_k:=B_k^{-1}(u)$ is a solution
to \eqref{3.w} with $a$ replaced by $a_k$. 

Problem \eqref{3.u} can be solved by the Galerkin method. Indeed, let $(v_i)
\subset W^{1,\infty}(\Omega)$ be an orthonormal basis of $H^1(\Omega)$ and
$X_n=\operatorname{span}\{v_1,\ldots,v_n\}$ for $n\in\N$.
The coefficients $\alpha_j$ of $u_n(x,t) = \sum_{j=1}^n \alpha_j(t)v_j(x)$ solve the
system of ordinary differential equations
\begin{equation}\label{3.ode}
  \sum_{j=1}^n A_{ij}(\alpha)\alpha_j' + \sum_{j=1}^n M_{ij} \alpha_j 
	= 0\quad\mbox{in }(0,T], \quad \alpha_j(0)=\alpha_{n,j}^0,
\end{equation}
where we assumed that $u_n^0:=\sum_{j=1}^n \alpha_{n,j}^0v_j(x)$ 
converges to $B_k(w^0)$
in $H^1(\Omega)$ as $n\to\infty$, and the matrices $(A_{ij}(\alpha))$ and $(M_{ij})$
are given by
$$
  A_{ij}(\alpha) = \int_\Omega\frac{v_iv_j}{a_k(B_k^{-1}(u_n))}dx, \quad
	M_{ij} = \int_\Omega\na v_i\cdot\na v_j dx.
$$
Since the coefficients $a_k\circ B_k^{-1}$ are bounded from below and above,
the properties of $(v_i)$ imply that $(A_{ij}(\alpha))$ is strictly positive definite.
Thus, \eqref{3.ode} possesses a global solution $\alpha\in C^1([0,T];\R^n)$
for any $T>0$.

For the limit $n\to\infty$, we need some uniform estimates. We multiply
\eqref{3.ode} by $\alpha_i'$ and sum over $i=1,\ldots,n$ leading to
$$
  \int_\Omega\frac{(\pa_t u_n)^2}{a_k(B_k^{-1}(u_n))}dx
	+ \frac12\frac{d}{dt}\int_\Omega|\na u_n|^2 dx = 0.
$$
Integration over $(0,t)$ for $t\le T$ yields the uniform bounds
$$
  \|\pa_t u_n\|_{L^2(Q_T)}\le C(k)\|u_n^0\|_{H^1(\Omega)}, \quad
	\sup_{0<t<T}\|\na u_n(t)\|_{L^2(\Omega)} \le \|u_n^0\|_{H^1(\Omega)}.
$$
The first estimate implies a uniform bound for $u_n$ in $L^2(Q_T)$ since
$u_n(t)=u_n(0)+\int_0^t \pa_t u_nds$.
By the Aubin--Lions lemma, there exists a subsequence which is not relabeled
such that $u_n\to u$ strongly in $L^2(Q_T)$, $\pa_t u_n\rightharpoonup \pa_t u$,
$\na u_n\rightharpoonup\na u$ weakly in $L^2(Q_T)$ as $n\to\infty$, 
and the limit $u$ satisfies
\begin{equation}\label{3.weaku}
  \int_0^T\int_\Omega \frac{\pa_t u\phi}{a_k(B_k^{-1}(u))}dxdt
	+ \int_0^T\int_\Omega\na u\cdot\na \phi dxdt = 0
\end{equation}
for all $\phi\in L^2(0,T;H^1(\Omega))$. Moreover, the initial condition
$u(0)=B_k(w^0)$ is fulfilled in $H^1(\Omega)$ and the bounds
\begin{equation}\label{3.estu}
  \|\pa_t u\|_{L^2(Q_T)}\le C(k)\|B_k(w^0)\|_{H^1(\Omega)}, \quad
	\sup_{0<t<T}\|\na u(t)\|_{L^2(\Omega)} \le \|B_k(w^0)\|_{H^1(\Omega)}
\end{equation}
hold. The function $w:=B_k^{-1}(u)\in H^1(\Omega)$ satisfies the equation
$$
  \int_0^T\int_\Omega\pa_t w\phi dxdt + \int_0^T\int_\Omega a_k(w)\na w\cdot
	\na\phi dxdt = 0
$$
for all $\phi\in L^2(0,T;H^1(\Omega))$ and the initial condition $w(0)=w^0$
in $H^1(\Omega)$. Thus, by the weak maximum principle, $\Lambda_*\le w\le\Lambda^*$
in $Q_T$ (see \eqref{2.Lambda} for the definition of $\Lambda_*$ and $\Lambda^*$).
Therefore, there exists $k_0\in\N$, depending only on $\Lambda_*$ and $\Lambda^*$ 
such that $a_{k_0}(w)=a(w)$ in $Q_T$. The function $B(s)=\int_0^s a(z)dz$ is
identical to $B_{k_0}$ on the range of $w$, since $a_{k_0}$ is identical to $a$
on the range of $w$. We conclude that $u_*\le u\le u^*$ in $Q_T$, where
$u_*:=B(\Lambda_*)$ and $u^*:=B(\Lambda^*)$.

{\em Step 2: Regularity.} We deduce from the bounds \eqref{3.estu} with $k=k_0$ 
and the lower and upper bounds for $u$ and $w$ that
$\pa_t w$ is bounded in $L^2(Q_T)$ and $\na w$ is bounded in 
$L^\infty(0,T;L^2(\Omega))$. Furthermore, by choosing $k=k_0$ and the
test function $\phi(x,t)=\zeta(t)\eta(x)$ in \eqref{3.weaku}, where $\zeta(t)$
approximates the delta distribution at $t$, we see that $u$ solves
\begin{equation}\label{3.weaku2}
  \int_\Omega\frac{\pa_t u\eta}{a(B^{-1}(u))}dx 
	+ \int_\Omega\na u(t)\cdot\na\eta dx = 0.
\end{equation}
Local regularity for elliptic equations \cite[Section 8.3, Theorem 8.9]{GiTr83}
implies that for any domain $\Omega'$ compactly embedded 
in $\Omega$, there exists $C(\Omega',u_*,u^*)>0$ such that
\begin{equation}\label{3.H2}
  \|D^2 u(t)\|_{L^2(\Omega')} \le C(\Omega',u_*,u^*)\|\pa_t u(t)\|_{L^2(\Omega)},
\end{equation}
which shows, using \eqref{3.estu}, that $u\in L^2(0,T;H^2(\Omega'))$.
Applying case (c) of Theorem 1.6 in \cite{Zan00} with $\alpha = 1$ therein,
there exists $\eps > 0$ depending only on the Lipschitz constant of $\pa\Omega$ 
such that for all $2/(1-\eps)\leq q < 3/(1-\eps)$, the gradient 
satisfies the global bound
\begin{equation}\label{3.W1q}
  \|\na u(t)\|_{L^q(\Omega)} \le C\|\pa_t u(t)\|_{L^{3q/(3+q)}(\Omega)}.
\end{equation}
Since $3q/(3+q)\le 2$ for $q\le 6$, we deduce that $u\in L^2(0,T;W^{1,q}(\Omega))$. 

We know from \eqref{3.estu} that $|\na u|^2$ is bounded in $L^\infty(0,T;L^1(\Omega))$.
Moreover, since $|\na u|\in L^2(0,T;H^1_{\rm loc}(\Omega))$, the Sobolev embedding
theorem yields $|\na u|\in L^2(0,T;L^6_{\rm loc}(\Omega))$ (recall that the
space dimension is three). Interpolation with $\theta=3/4$ then shows for $z=|\na u|^2$
and any $\Omega'$ compactly embedded in $\Omega$ that
\begin{equation}\label{3.L43}
  \|z\|_{L^{4/3}(0,T;L^2(\Omega'))}^{4/3}
	\le \int_0^T\|v\|_{L^3(\Omega')}^{4\theta/3}\|z\|_{L^1(\Omega')}^{4(1-\theta)/3}dt
	\le \|z\|_{L^\infty(0,T;L^1(\Omega'))}^{1/3} \int_0^T\|z\|_{L^3(\Omega')}dt
\end{equation}
and consequently $z=|\na u|^2\in L^{4/3}(0,T;L^2_{\rm loc}(\Omega))$.
In order to show a global bound, we combine the regularity properties
$|\na u|^2\in L^\infty(0,T;L^1(\Omega))$, $|\na u|^2\in L^1(0,T;L^{q/2}(\Omega))$
and interpolate with $r=2-2/q$ and $\theta=1/r$ such that
$1/r = (1-\theta)/1+2\theta/q$:
$$
  \|z\|_{L^r(Q_T)}^r 
	\le \int_0^T \|z\|_{L^1(\Omega)}^{r(1-\theta)}\|z\|_{L^{q/2}(\Omega)}^{r\theta}dt
	\le \|z\|_{L^\infty(0,T;L^1(\Omega))}^{r(1-\theta)}
	\int_0^T\|z\|_{L^{q/2}(\Omega)}dt
$$
and hence, $z=|\na u|^2\in L^{2-2/q}(Q_T)$. 

Finally, we consider $v=-\kappa G'(w)\na w$. Introducing 
$f(u):=-\kappa G'(B^{-1}(u))(B^{-1})'(u)$ and recalling that $w=B^{-1}(u)$, 
this means that $v=f(u)\na u$.
The positive lower bound for $u$ and estimate \eqref{3.W1q} then show that
$v\in L^2(0,T;L^q(\Omega))$. It holds that
$$
  \na v = f(u)D^2u+f'(u)\na u\otimes\na u, \quad
	\diver v = f(u)\Delta u + f'(u)|\na v|^2.
$$
Using estimates \eqref{3.H2} and \eqref{3.L43} and the positive lower bound for $u$,
it follows that $|\na v|\in L^{4/3}(0,T;L^2_{\rm loc}(\Omega))$. Furthermore,
$\Delta u = \pa_t u/(a(B^{-1}(u)))\in L^2(Q_T)$ and $|\na u|^2\in L^{2-2/q}(Q_T)$
gives $\diver v\in L^{2-2/q}(Q_T)$.

{\em Step 3: Proof of $v\cdot\na\phi_m\to 0$ strongly in $L^1(Q_T)$.}
Proposition 3.7 in \cite{ABDG98} shows that for some $\delta>0$, it holds that
\begin{equation}\label{3.H12}
  \|\na u(t)\|_{H^{1/2+\delta}(\Omega)} 
	\le C(\Omega,u_*,u^*)\|\pa_t u(t)\|_{L^2(\Omega)}.
\end{equation}
It follows that $\na u(t)$ is defined in $L^2(\pa\Omega)$. Then the weak formulation
\eqref{3.weaku2} implies that $u$ satisfies the Neumann boundary conditions
$\na u(t)\cdot\nu=0$ in $L^2(S_i)$ for $i=1,\ldots,\ell$. Recalling that
$v=f(u)\na u$ and $f(u)$ is globally bounded in $Q_T$, we infer that
$|v\cdot\na\phi_m|\le \|f(u)\|_{L^\infty(Q_T)}|\na u\cdot\na\phi_m|$.
Thus, by Lemma \ref{lem.approx} and estimate \eqref{3.H12}, for $1\le p<1/(1-\delta)$,
\begin{align*}
  \|v\cdot\na\phi_m\|_{L^2(0,T;L^p(\Omega))}
	&\le C(\delta)\|f(u)\|_{L^\infty(Q_T)}\|\na u\|_{L^2(0,T;H^{1/2+\delta}(\Omega))} \\
	&\le C(\delta)\|f(u)\|_{L^\infty(Q_T)}\|\pa_t u\|_{L^2(Q_T)} \le C.
\end{align*}
Again by Lemma \ref{lem.approx}, we have $\na\phi_m\to 0$ a.e.\ in $\Omega$
and in view of the previous estimate, $v\cdot\na\phi_m\to 0$ a.e.\ in $Q_T$.
By dominated convergence, this convergence also holds in $L^1(Q_T)$. This
finishes the proof.
\end{proof}

The next step is to approximate the velocity $v=-\kappa G'(w)\na w$ by smooth
vector fields.

\begin{lemma}[Approximation of $v$]\label{lem.v}
There exists a family of smooth vector fields $v_\eps$ for $\eps>0$ satisfying
$v_\eps\cdot\nu=0$ on $\pa\Omega\times(0,T)$ such that, as $\eps\to 0$,
$$
  v_\eps\to v\quad \mbox{strongly in }L^2(Q_T), \quad
	\diver v_\eps\to \diver v\quad\mbox{strongly in }L^1(Q_T).
$$
\end{lemma}

\begin{proof}
Let $(\phi_m)$ be the family constructed in Lemma \ref{lem.approx} and define
$\psi_\eps:=\phi_{2/\eps}$ for $\eps>0$. Then $\phi_\eps$ takes values in $[0,1]$ and 
$\psi_\eps(x)=0$ for all $x\in\Omega$ such that 
$\operatorname{dist}(x,\pa\Omega)<\eps/2$. Furthermore, let 
$\Phi\in C^\infty(\overline{B_1(0)})$ satisfy $\int_{B_1(0)}\Phi(z)dz=1$,
where $B_1(0)$ is the unit ball in $\R^3$. Then we define
$$
  v_\eps(x,t) = \int_{B_1(0)}\Phi(z)v(x+\eps z,t)\psi_\eps(x+\eps z)dz.
$$
Since $\psi_\eps \to 1$ uniformly on every compact subset of $\Omega$ as $\eps\to 0$, 
by Lemma \ref{lem.approx}, we have $v_\eps\to v$ strongly in $L^2(Q_T)$. Moreover,
\begin{align*}
  \diver v_\eps(x,t) &= \int_{B_1(0)}\Phi(z)\diver v(x+\eps z,t)\psi_\eps(x+\eps z)dz \\
	&\phantom{xx}{}+ \int_{B_1(0)}\Phi(z)v(x+\eps z)\cdot\na\psi_\eps(x+\eps z)dz.
\end{align*}
The first integral converges to $\diver v$ strongly in $L^{2-2/q}(Q_T)$ since
$\psi_\eps\nearrow 1$ in $\Omega$, and the second integral converges to zero
strongly in $L^1(Q_T)$, due to the fact that $v\cdot\na\phi_{2/\eps}\to 0$ strongly
in $L^1(Q_T)$ by Lemma \ref{prop.v}. 
\end{proof}

We also approximate the initial densities: Let $\rho^{0,\eps}\in 
C^1(\overline\Omega;\R^N)$ for $0<\eps<1$ be approximations of $\rho^0$ such that
\begin{align*}
  & (1-\eps)\inf_{y\in\Omega}\rho_i^0(y)\le \rho_i^{0,\eps}(x)
	\le (1+\eps)\sup_{y\in\Omega}\rho_i^0(y), \quad x\in\Omega, \\
  & \rho^{0,\eps}\to\rho^0\quad\mbox{strongly in }L^2(\Omega;\R^N)\mbox{ as }\eps\to 0.
\end{align*}
Furthermore, we use the characteristic curves $\Phi_\eps(s;x,t)$ associated to $v_\eps$,
$$
  \Phi_\eps'(s;t,x) = v_\eps(\Phi(s;t,x),s)\quad\mbox{for }s\in(0,T), \quad
	\Phi(t;t,x) = x
$$
(see Section \ref{sec.class}), to solve 
\begin{equation*}%\label{3.ueps}
  \pa_t u^\eps_i + v_\eps\cdot\na u^\eps_i = 0\quad\mbox{in }Q_T,
	\quad u^\eps_i(0) = u_i^{0,\eps} := \frac{\rho_i^{0,\eps}}{\Lambda(\rho^{0,\eps})},
	\quad i=1,\ldots,N.
\end{equation*}
These equations correspond to \eqref{1.u} but with the velocity $v$ replaced
by $v_\eps$. They can be solved explicitly in terms of the characteristic curves, 
\begin{equation}\label{3.ueps}
  u_\eps(x,t) = u^{0,\eps}(\Phi_\eps(0;t,x)) 
	= \frac{\rho^{0,\eps}(\Phi_\eps(0;t,x))}{\Lambda(\rho^{0,\eps}(\Phi_\eps(0;t,x)))},
	\quad (x,t)\in Q_T,
\end{equation}
which shows that $u_\eps$ belongs to $C^1(\overline{Q}_T;\R^N)$ and $\Lambda(u_\eps)=1$
in $Q_T$. Furthermore, we deduce from the growth condition in Assumption (A4) that
$$
  r_0(1-\eps)\sum_{i=1}^N\inf_{y\in\Omega}\rho_i^0(y)
	\le \Lambda(\rho^{0,\eps}(x)) \le r_1(1+\eps)\sum_{i=1}^N\sup_{y\in\Omega}\rho_i^0(y)
$$
for $x\in\Omega$ and consequently
\begin{equation}\label{3.bueps}
  \frac{1-\eps}{r_1(1+\eps)}\frac{\inf_{y\in\Omega}\rho_i^0(y)}{\sum_{i=1}^N
	\sup_{y\in\Omega}\rho_i^0(y)} \le u_i^\eps(x,t)
	\le \frac{1+\eps}{r_0(1-\eps)}\frac{\sup_{y\in\Omega}\rho_i^0(y)}{\sum_{i=1}^N
	\inf_{y\in\Omega}\rho_i^0(y)}
\end{equation}
for $(x,t)\in Q_T$. 

Next, we define the approximate densities $\rho_i^\eps:=w u_i^\eps$, 
where $i=1,\ldots,N$ and $w$ is the solution to the porous-medium-type equation
\eqref{1.w} from Proposition \ref{prop.v}. We prove that a subsequence of
$\rho_i^\eps$ converges to a renormalized solution to the mass continuity
equation \eqref{1.cont}. For the following result, we recall definition 
\eqref{1.W12} of the space $W_2^1(Q_T)$.

\begin{proposition}[Convergence of $\rho_i^\eps$]\label{prop.rhoeps}
The family $(\rho^\eps)\subset W_2^1(Q_T;\R^N)\cap L^\infty(Q_T;\R^N)$ is bounded
in $L^\infty(Q_T;\R^N)$ and there exists $c_1>0$ such that for all $0<\eps<\eps_0$,
$$
  \inf_{(x,t)\in Q_T}\rho_i^\eps(x,t)\ge c_1, \quad i=1,\ldots,N.
$$
There exists a subsequence of $(\rho^\eps)$ (not relabeled) such that
$\rho^\eps\to\rho$ strongly in $L^2(Q_T;\R^N)$. 
The limit $\rho=(\rho_1,\ldots,\rho_n)$ satisfies 
$\rho\in L^\infty(Q_T;\R^N)$ with $\rho_i\ge c_1$ in $Q_T$, 
and $\rho$ is a renormalized solution to \eqref{1.cont}, i.e.,
\begin{align*}
  -\int_0^T\int_\Omega & b(\rho)\pa_t\zeta dxdt
	- \int_0^T\int_\Omega b(\rho)\na\zeta\cdot v dxdt \\
	&= \int_\Omega b(\rho^0)\zeta(0)dx 
	- \int_0^T\int_\Omega\big(b'(\rho)\cdot\rho-b(\rho)\big)\zeta\diver v dxdt
\end{align*}
is satisfied for all $\zeta\in C_0^1(\overline{Q}_T)$ 
and $b\in C^1(\R^N)$, where $v=-\kappa G'(w)\na w$ and $w$ is the weak solution to
\eqref{1.w}--\eqref{1.bicw}.
\end{proposition}

\begin{proof}
We know from Proposition \ref{prop.v} that $w\in W_2^1(Q_T)$ and from
\eqref{3.ueps} that $u^\eps\in C^1(\overline\Omega;\R^N)$, which shows that
$\rho^\eps=wu^\eps\in H^1(\Omega;\R^N)$. Moreover, 
since $w$ satisfies the uniform lower 
and upper bounds \eqref{2.Lambda} and $u^\eps$ satisfies the uniform bounds 
\eqref{3.bueps}, also $\rho^\eps_i$ is uniformly bounded from below and above.
At time $t=0$, we have
\begin{equation}\label{3.rhoeps0}
  \rho_i^\eps(0) = w^0u_i^{0,\eps} 
	= \Lambda(\rho^0)\frac{\rho_i^{0,\eps}}{\Lambda(\rho^{0,\eps})} 
	=: \widetilde{\rho}_i^{0,\eps},
\end{equation}
and the strong convergence $\rho^{0,\eps}\to\rho^0$ in $L^1(\Omega;\R^N)$
implies that $\rho_i^\eps(0)\to \rho^0$ strongly in $L^1(\Omega;\R^N)$. 
As $u^\eps$ solves \eqref{3.ueps} and $w$ solves $\pa_t w + \diver(vw) = 0$, 
it follows that $\rho_i^\eps=wu_i^\eps$ solves
\begin{align*}
  \pa_t\rho_i^\eps + v^\eps\cdot\na\rho_i^\eps
	&= w\big(\pa_t u_i^\eps + v^\eps\cdot\na u_i^\eps\big)
	+ u_i^\eps\big(\pa_t w + v^\eps\cdot\na w\big) \\
	&= u_i^\eps\big(\pa_t w + (v^\eps-v)\cdot\na w + \diver(vw) - w\diver v\big) \\
	&= u_i^\eps(v^\eps-v)\cdot\na w - u_i^\eps w\diver v
	= u_i^\eps\diver((v^\eps-v)w) - u_i^\eps w\diver v^\eps \\
	&= u_i^\eps\diver((v^\eps-v)w) - \rho_i^\eps\diver v^\eps.
\end{align*}
In view of the regularity $w\in W_2^1(Q_T)$, $u^\eps\in C^1(\overline{Q}_T;\R^N)$,
and $v\in L^{4/3}(0,T;H^1_{\rm loc}(\Omega))$ (see Proposition \ref{prop.v}),
these calculations make sense a.e.\ in $\Omega'\times(0,T)$ for domains
$\Omega'$ compactly embedded in $\Omega$. We infer that
$$
  \pa_t\rho_i^\eps + \diver(\rho_i^\eps v^\eps) = r_i^\eps 
	:= u_i^\eps\diver((v^\eps-v)w).
$$

Using $\psi\in C_0^1(\overline\Omega\times[0,T);\R^N)$ as a test function
and integrating by parts, we see that
\begin{align}
  -\int_0^T&\int_\Omega\rho^\eps\cdot\pa_t\psi dxdt
	- \int_0^T\int_\Omega(\rho^\eps)^\top\na\psi v^\eps dxdt \nonumber \\
	&= \int_\Omega\widetilde\rho^{0,\eps}(x)\cdot\psi(x,0)dx
	+ \int_0^T\int_\Omega r^\eps\cdot\psi dxdt, \label{3.rhoeps}
\end{align}
where $\widetilde\rho_i^{0,\eps}$ is defined in \eqref{3.rhoeps0}. We wish to
pass to the limit $\eps\to 0$. We deduce from H\"older's inequality and Lemma 
\ref{lem.v} that
\begin{align*}
  \|r^\eps\|_{L^1(Q_T)}
	&= \|\rho^\eps\diver(v^\eps-v) + u^\eps(v^\eps-v)\cdot\na w\|_{L^1(Q_T)} \\
	&\le \|\rho^\eps\|_{L^\infty(Q_T)}\|\diver(v^\eps-v)\|_{L^1(Q_T)}
	+ \|u^\eps\|_{L^\infty(Q_T)}\|v^\eps-v\|_{L^2(Q_T)}\|\na w\|_{L^2(Q_T)} \\
	&\to 0\quad\mbox{as }\eps\to 0.
\end{align*}
Since $(\rho^\eps)$ is uniformly bounded in $L^\infty(Q_T)$, there exists a 
subsequence of $(\rho^\eps)$ which is not relabeled such that
$\rho^\eps\rightharpoonup^*\rho$ weakly* in $L^\infty(0,T;L^\infty(\Omega))$ 
and $\rho\in L^\infty(Q_T)$
satisfies $\rho_i\ge c_1$ in $Q_T$, $i=1,\ldots,N$. 
Thus, taking into account the convergence
in Lemma \ref{lem.v}, $\rho_i^\eps v_j^\eps\rightharpoonup \rho_i v_j$ weakly
in $L^2(Q_T)$. In view of $\rho_i^\eps(0)\to \rho^0$ strongly in $L^2(\Omega;\R^N)$,
the limit $\eps\to 0$ in \eqref{3.rhoeps} leads to
\begin{equation}\label{3.rho}
  -\int_0^T\int_\Omega\rho\cdot\pa_t\psi dxdt
	- \int_0^T\int_\Omega\rho^\top\na\psi v dxdt
	= \int_\Omega\rho^0(x)\cdot\psi(x,0)dx
\end{equation}
for all $\psi\in C_0^1(\overline\Omega\times[0,T);\R^N)$, recalling that
$v=-\kappa G'(w)\na w$.

It remains to show that $\Lambda(\rho)=w$. To this end, we rely on techniques
of renormalization for the continuity equation. Let $\Omega'$ be compactly embedded
in $\Omega$ and let $\phi_m$ be the function constructed in Lemma \ref{lem.approx}
satisfying $\phi_m=1$ in $\Omega'$. (This is possible since $\phi_m\to 1$ locally
uniformly in $\Omega$ as $m\to 0$.) The function $\widetilde\rho:=\rho\phi_m$
is nonnegative and compactly supported in $\Omega\times[0,T]$. Replacing $\psi$ by
$\psi\phi_m$ in \eqref{3.rho} yields
\begin{align*}
  -\int_0^T&\int_\Omega\widetilde\rho\cdot\pa_t\psi dxdt 
	- \int_0^T\int_\Omega\widetilde\rho^\top \na\psi v dxdt \\
	&= \int_\Omega\rho^0(x)\phi_m(x)\cdot\psi(x,0)dx 
	+ \int_0^T\int_\Omega(\rho\cdot\psi)(v\cdot\na\phi_m)dxdt.
\end{align*}
This is the weak formulation of
$$
  \pa_t\widetilde\rho + \diver(\widetilde \rho v) = \rho(v\cdot\na\phi_m)
	\quad\mbox{in }\mathcal{D}'(Q_T).
$$
In view of the regularity $\widetilde\rho\in L^\infty(Q_T)$, 
$v\in L^{4/3}(0,T;H^1_{\rm loc}(\Omega))$, and $\rho(v\cdot\na\phi_m)
\in L^{4/3}(0,T;$ $L^2_{\rm loc}(\Omega))$, we can apply the theory of renormalized
solutions to the continuity equation (see, e.g., \cite[Theorem 10.29]{FeNo09})
to conclude that
$$
  \pa_t b(\widetilde\rho) + \diver(b(\widetilde\rho)v)
	+ \big(b'(\widetilde\rho)\cdot\widetilde\rho-b(\widetilde\rho)\big)\diver v
	= (\rho \cdot b^{\prime}(\widetilde{\rho}))(v\cdot\na\phi_m)
	\quad\mbox{in }\mathcal{D}'(Q_T)
$$
for any $b\in C^1(\R_{+,0}^N)\cap W^{1,\infty}(\R_+^N)$. Using the regularity 
$\diver v\in L^{2-2/q}(Q_T)$ for $3 < q\le 6$ proved in Proposition
\ref{prop.v}, we can formulate this identity in the weak form
\begin{align}
  -\int_0^T&\int_\Omega b(\rho\phi_m)\pa_t\zeta dxdt
	- \int_0^T\int_\Omega b(\rho\phi_m)v\cdot\na\zeta dxdt \nonumber \\
	&= \int_\Omega b(\rho^0(x)\phi_m(x))\zeta(x,0)dx 
	- \int_0^T\int_\Omega\big(b'(\rho\phi_m)\cdot\rho\phi_m
	-b(\rho\phi_m)\big)(\diver v)\zeta dxdt \nonumber \\
	& + \int_{0}^T \int_{\Omega}\rho \cdot b^{\prime}(\widetilde{\rho}) 
	(v\cdot\na\phi_m) \zeta dxdt \label{3.tilderho}
\end{align}
for all $\zeta\in C^1(\overline{Q}_T)$ such that $\zeta(T)=0$. We know that
$\phi_m\to 1$ locally uniformly in $\Omega$ as $m\to\infty$, 
by Lemma \ref{lem.approx}, and 
$v\cdot\na\phi_m\to 0$ strongly in $L^1(Q_T)$, by Proposition \ref{prop.v}.
Thus, passing to the limit $m\to\infty$ in \eqref{3.tilderho}, we obtain
\begin{align}
  -\int_0^T&\int_\Omega b(\rho)\pa_t\zeta dxdt
	- \int_0^T\int_\Omega b(\rho)v\cdot\na\zeta dxdt \nonumber \\
	&= \int_\Omega b(\rho^0(x))\zeta(x,0)dx 
	- \int_0^T\int_\Omega\big(b'(\rho)\cdot\rho
	-b(\rho)\big)(\diver v)\zeta dxdt. \label{3.weakrho}
\end{align}

Finally, we verify that $\rho$ is in fact a weak solution. For this, we need
to prove that $w=\Lambda(\rho)$. Since $\Lambda$ is continuously differentiable
and the range of $\rho$ is contained in a compact subset of $\R_{+}^N$, we can choose
a $C^1$ function $b$ which is equal to $\Lambda$ on this range. Thus, $b=\Lambda$
is admissible in \eqref{3.weakrho}. Then we deduce from the positive homogeneity
of $\Lambda$, i.e.\ $\Lambda'(\rho)\cdot\rho-\Lambda(\rho)=0$, that
$$
  -\int_0^T\int_\Omega\Lambda(\rho)\pa_t\zeta dxdt
	- \int_0^T\int_\Omega\Lambda(\rho)v\cdot\na\zeta dxdt
	= \int_\Omega\Lambda(\rho^0(x))\zeta(x,0)dx
$$
for all $\zeta\in C^1(\overline{Q}_T)$ such that $\zeta(T)=0$. The function
$w$ is a solution to \eqref{1.w}, so $z:=w-\Lambda(\rho)$ satisfies 
$$
  -\int_0^T\int_\Omega z\pa_t\zeta dxdt - \int_0^T\int_\Omega zv\cdot\na\zeta dxdt
	= \int_0^T\int_\Omega z(x,0)\zeta(x,0)dx = 0.
$$
This means that $\pa_t z + \diver(zv)=0$ in $\mathcal{D}'(Q_T)$, $z(0)=0$ in
$\Omega$. Applying renormalization theory again,
we see that $z=0$ in $Q_T$. Hence, $w=\Lambda(\rho)$, and $\rho$ is a weak solution 
to \eqref{1.cont}.
\end{proof}

%%%%%%%%%%%%%%%%%%%%%%%%%%%%%%%%%%%%%%%%%%%%%%%%%%%%%%%%%%%%%%%%%%%%%%%%%%%%%%%

\section{Extensions and open problems}\label{sec.open}

In this paper, we have deliberately kept the technicality rather low to highlight
the key ideas. In this section, we briefly point out some direct extensions of 
our results as well as some open problems.

\subsection{Reaction terms} 

Let us consider the mass balance equation \eqref{1.cont} with nonvanishing
right-hand sides $r_i = r_i(\rho)$ for $i=1,\ldots,N$, modelling some bulk source 
of mass, possibly from chemical reactions. 
Under suitable technical restrictions on the data, we expect the existence and 
uniqueness of a classical solution as in Theorem \ref{thm.class}. We briefly 
sketch the arguments.

For the variables $w :=\Lambda(\rho)$ and $u := \rho/\Lambda(\rho)$ 
introduced in Section \ref{keyideas}, we define 
\begin{equation}\label{4.f}
  f(w,u) := r(wu) \cdot \Lambda'(u), \quad 
	g_i(w,u) := \frac{1}{w}(r_i(wu) - u_if(w,u)), \quad i=1,\ldots,N.
\end{equation}
Equations \eqref{1.w} for $w$ and \eqref{1.u} for $u$ change to
\begin{align}\label{4.reacw}
  \pa_t w - \diver(\kappa(w) w G'(w)\na w) = f(w,u) 
	&\quad\mbox{in }Q_T ,\\
  \dot{u}_i =  g_i(w,u) &\quad\mbox{in }Q_T,\ i=1,\ldots,N, \label{4.reacu}
\end{align}
where $\dot{u} = \partial_t u + v \cdot \nabla u$ denotes the material derivative 
with velocity $v = -\kappa(w)\nabla G(w)$.
  
As exposed in our key ideas, the main ingredient to prove the well-posedness of such
systems is the maximum principle for \eqref{4.reacw}. We will not discuss
all possible growth conditions needed for $\Lambda$, $\kappa$, $G$ and $r$ such that 
the weak maximum principle in \eqref{4.reacw} is valid. Instead, we only discuss
the example of a sufficiently smooth field $r \in C^{1}(\R^N_{+,0})$ under
the following assumptions:

\begin{labeling}{(A44)}
\item[(A5)] There exist $0 < w_* < w^* < + \infty$ such for all $u \in \R^N_+$ 
satisfying $\Lambda(u) \leq 1$, we have $f(w,u) \geq 0$ for all $w \leq w_*$ 
and $f(w,u) \leq 0$ for all $w \geq w^*$;
\item[(A6)] The modified functions $\widetilde{g}_i(w,u) := (1/w)(r_i(wu) 
- u_if(wu)/\Lambda(u))$ are quasi-positive for $i=1,\ldots,N$, i.e., for all 
$w \in \R^N_+$, we have $\widetilde{g}_i(w,u) \geq 0$ whenever $u_i = 0$. 
\end{labeling} 

We claim that under Assumptions (A1)--(A6), there exist classical solutions
to \eqref{4.reacw}, \eqref{4.reacu} with initial and boundary conditions
\eqref{1.bicw} and \eqref{1.u}. We use a fixed-point argument. For this,
let $\bar{u} \in L^{\infty}(Q_T;\R^N_{+,0})$ be given such that
$\Lambda(\bar{u}) \leq 1$. Then there exists a weak solution $w$ to
\begin{align*}
  & \pa_t w - \diver(\kappa(w)wG'(w)\na w)=f(w,\bar{u})\quad\mbox{in }Q_T, \\
	& \na w\cdot\nu=0\quad\mbox{on }\pa\Omega,\ t>0, \quad w(0)=w^0\quad\mbox{in }\Omega.
\end{align*}
Using the test functions $(w - w^*)^+=\max\{0,w-w^*\}$ and 
$(w-w_*)^-=\min\{0,w-w_*\}$, Assumption (A5) yields
lower and upper bounds for $w$. The Lipschitz continuity of $f$ allows us to
conclude the uniqueness of the solution.

Next, we use parabolic regularity for quasilinear equations to derive an
intermediate bound for $\na w$. For Neumann problems, the global gradient bound 
for $\nabla w$ in $L^{\infty}(Q_T)$ is proved, for instance, in 
\cite[Chapter XIII, Section 2.2]{Lie96}. 
This bound only requires an $L^{\infty}$ bound on the source term 
$f(w,\bar{u})$ in $Q_T$. Using Assumption (A4) on $\Lambda$ and the bounds 
$r_0|\bar{u}| \leq \Lambda(\bar{u}) \leq 1$, the source term can be estimated in 
$L^{\infty}$ independently on $\bar{u}$ (and, of course, independently of $w$). 

Arguing like in \cite{BHIM12}, we invoke linear parabolic theory: We obtain 
the H\"older continuity of $w$ \cite[Chapter V, Theorem 7.1]{LSU68} and, for every 
$1\leq p < \infty$, the estimates 
$$
  \|w\|_{W^{2,1}_p(Q_T} \leq C\big(\|w_0\|_{W^{2-2/p,p}(\Omega)} 
  + \|f(w,\bar{u})\|_{L^{p}(Q_T)}\big),
$$
where $W^{2,1}_p(Q_T)=L^p(0,T;W^{2,p}(\Omega))\cap W^{1,p}(0,T;L^p(\Omega))$
\cite[Chapter 5, \S 9]{LSU68}.
In three space dimensions, the embedding properties of $W^{2,1}_p(Q_T)$ 
yield a bound for $\nabla w$ in $C^{\beta,\beta/2}(\overline{Q}_T)$ with 
$\beta = 1-5/p>0$ for $p>5$. 

Now, we can introduce the fractions. It is possible to solve globally the 
characteristic differential equations
$$
  \Phi'(s;t,x) = v(\Phi(s;t,x),s),\quad s\in(0,T),\quad \Phi(t;t,x) = x.
$$
For all $s\in[0,T]$, the map $(x,t)\mapsto \Phi(s;x,t)$ belongs to 
$C^{1+\beta,1+\beta/2}(\overline{Q}_T)$. 
Since the map $u \mapsto \widetilde{g}_i(w,u)$ is Lipschitz continuous, we can 
solve the ODEs
\begin{equation}\label{4.ode}
  \dot{u}_i = \widetilde{g}_i(w,u), \quad t>0,\ i=1,\ldots,N.
\end{equation}
Thanks to Assumption (A6), the solutions $u_i$ remain globally positive. 
The definition of $\widetilde{g}$ in Assumption (A6) implies that 
$$
  \sum_{i=1}^N \widetilde{g}_i(w,u)\frac{\pa\Lambda}{\pa\rho_i}(u) 
	= \frac{1}{w}\bigg(\sum_{i=1}^N  r_i(wu) \frac{\pa\Lambda}{\pa\rho_i}(u) 
	- \frac{\Lambda(u)}{\Lambda(u)}f(w,u)\bigg) = 0.
$$
Thus, the solutions to \eqref{4.ode} satisfy 
$\dot{\Lambda}(u) = \sum_{i=1}^N \dot{u}_i\pa\Lambda/\pa\rho_i=0$, and 
$\Lambda(u) = 1$ is conserved. 

It is readily seen that $\bar{u} \mapsto u$ maps the set 
$M := \{u \in L^{\infty}(Q_T;\R^N) : 0 \leq u,\,\Lambda(u) \leq 1\}$ into itself. 
Moreover, the solution formula for \eqref{4.ode},
$$
  u(x,t) = u^0(\Phi(0;x,t)) + \int_{0}^t \widetilde{g}\big(w(\Phi(s;x,t),s),
	u(\Phi(s;x,t),s)\big)ds,
$$
can be used to prove a bound for $u$ in $C^1(\overline{Q}_T)$. Thus, 
$\bar{u} \mapsto u$ is compact in $L^{\infty}(Q_T)$, and the Schauder fixed-point 
theorem gives the existence of a fixed point. 

This argument can be extended to weak solutions and to more general assumptions
than (A5)--(A6), but we leave the details to the reader.

%%%%%%%%%%%%%%%%%

\subsection{External forces} 

A second problem is the presence of external forces in \eqref{1.v}. 
In the simplest case, the flow is subject to gravity, yielding
\begin{equation}\label{4.v}
  v = \kappa(-\nabla p + \rho_{\text{tot}}\vec{g}), \quad 
	\rho_{\text{tot}} = \sum_{i=1}^N \rho_i,
\end{equation}
where $\rho_{\text{tot}}$ is the total mass density and $\vec{g}$ is the 
constant vector of earth gravitational acceleration. In terms of the variables $w$
and $u$, we find that $\rho_{\text{tot}} = w\sum_{i=1}^N u_i$. After the change of 
variables, equation \eqref{1.w} for $w := \Lambda(\rho)$ becomes
\begin{equation}\label{4.grav}
  \pa_t w - \diver\bigg(\kappa w \bigg(G'(w)\na w - \sum_{i=1}^N u_iw\vec{g}
	\bigg)\bigg) = 0 \quad\mbox{in }Q_T. 
\end{equation}
The presence of the contribution $\sum_{i=1}^N u_i$ in the flux prevents the 
higher regularity of $w$. We therefore expect that the extension to the case 
of nontrivial external forces is a challenging open problem. 

However, if the sum of initial fractions is constant, 
$\sum_{i=1}^N \rho_i^0/\Lambda(\rho^0) = C_0$ in $\Omega$ for some $C_0>0$,
and if the fractions are simply transported, we may replace \eqref{4.grav} by 
$\pa_t w - \diver(\kappa w(G'(w)\na w - C_0w\vec{g}) ) = 0$, yielding
a problem which can be treated by the previous arguments.

%%%%%%%%%%%%%%%%%%%%%%

\subsection{Other boundary conditions}\label{sec.bc}

The last problem that we would like to discuss is the choice of boundary conditions. 
Instead of \eqref{1.bic}, we might impose the Dirichlet condition $p = p_0$ on 
$\pa\Omega \times (0,T)$, where $p_0$ is a given function. This type of pressure 
boundary condition corresponds to a free in-outflow problem. For instance, we can 
think of the domain $\Omega$ as a fixed control region in a larger environment 
occupied by the fluid. 

In formulating this problem, we realize that equations \eqref{1.cont} are not 
well posed on $\Omega$. In the absence of the impermeability condition \eqref{1.bic}, 
the trajectories of the characteristics are clearly not confined to $\Omega$. In 
order to solve this kind of boundary-value problems, we need a representation of 
the flow outside of $\Omega$. Mathematically, we need an extension operator $E$ 
which, for each velocity field given on $\pa\Omega$, provides its extension $E(v)$ 
to a larger region $\widetilde{\Omega}$. Moreover, as the trajectories are not 
confined to $\Omega$, the initial state needs to be known in the larger region 
$\widetilde{\Omega}$.

The strategy to solve the free flow problem is again to solve the equations
\begin{equation}\label{4.freeflow}
  \pa_t w - \diver(\kappa w G'(w)\na w ) = 0 \quad\mbox{in }Q_T, \quad 
	w = G^{-1}(p_0) \quad\mbox{on } \pa\Omega,\ t>0,
\end{equation}
with initial conditions $w(0) = w^0$ in $\Omega$. 
We solve the differential equations for the characteristics in the larger domain 
$\widetilde{\Omega} \times (0,T)$ with the extended velocity field 
$v = E(-\kappa G^{\prime}(w) \na w)$. Since the flow is confined to a bounded region, 
it is natural to assume that $E(v)$ possesses compact support, i.e., it vanishes 
uniformly outside of the domain $\widetilde{\Omega}$.

Then, we transport the fractions via $u(x,t) = u^0(\Phi(0;x,t))$ for 
$x \in \widetilde{\Omega}$ and $t > 0$. It is readily verified that 
$\rho_i(x,t) := (w u_i)(x,t)$ solves $\pa_t \rho_i + \diver(\rho_i v) = 0$ in 
$Q_T$ and $\rho_i(0,x) = \rho_i^0(x)$ for $x \in \Omega$. 

We will discuss these ideas in more detail in an upcoming publication devoted 
to the optimal control of this type of flow problems.

%%%%%%%%%%%%%%%%%%%%%%%%%%%%%%%%%%%%%%%%%%%%%%%%%%%%%%%%%%%%%%%%%%%%%%%%%%%%%

\end{document}